\documentclass[12pt]{article}
\usepackage{latexsym,amsmath,amsthm,amssymb}
\usepackage{a4wide}
\usepackage{hyperref}
\usepackage{marginnote}
\usepackage{color}

\theoremstyle{plain}
\newtheorem{theorem}{Theorem}[section]
\newtheorem{proposition}[theorem]{Proposition}
\newtheorem{lemma}[theorem]{Lemma}

\theoremstyle{definition}

\theoremstyle{remark}

\newcommand{\ind}{\mathrm{1}\hskip -3.2pt \mathrm{I}} 
\newcommand{\id}{\mbox{\rm Id}}

\def\R{\mathbb R}
\def\N{\mathbb N}

\def\al{\alpha}
\def\om{\omega}
\def\Om{\Omega}
\def\ga{\gamma}
\def\ge{\geq}
\def\lt{\left}
\def\rt{\right}

\def\x{\overline x}
\def\tr{\mathop{\rm Tr}\nolimits}
\def\d{\mathop{\rm div}\nolimits}
\def\det{\mathop{\rm det}\nolimits}
\def\vol{\mathop{\rm vol}\nolimits}
\def\BV{\mathop{\rm BV}\nolimits}
\def\supp{\mathop{\rm supp}\nolimits}

\numberwithin{equation}{section}

\title{Sharp weighted Sobolev and Gagliardo-Nirenberg inequalities on half spaces via mass transport and consequences}
\author{Van Hoang Nguyen\footnote{Institut de Math\'ematiques de Jussieu, Universit\'e Pierre et Marie Curie (Paris 6), 4 place Jussieu, 75252 Paris, France. Email: vanhoang@math.jussieu.fr, vanhoang0610@yahoo.com.}}

\begin{document}
\maketitle


\renewcommand{\thefootnote}{}

\footnote{Supported by the French ANR GeMeCoD, ANR 2011 BS01 007 01.}

\footnote{2010 \emph{Mathematics Subject Classification\text}: 26D15, 46E35.}

\footnote{\emph{Key words and phrases\text}: Weighted Sobolev inequalities, weighted Gagliardo-Nirenberg inequalities, weighted $L^p$-logarithmic Sobolev inequalities, sharp constants, Brenier map.}

\renewcommand{\thefootnote}{\arabic{footnote}}
\setcounter{footnote}{0}

\begin{abstract}
By adapting the mass transportation technique of Cordero-Erausquin, Nazaret and Villani, we obtain a family of sharp  Sobolev  and  Gagliardo-Nirenberg (GN) inequalities on the half space $\R^{n-1}\times\R_+$, $n\geq 1$ equipped  with the weight $\omega(x) = x_n^a$, $a\geq 0$. It amounts to work with the fractional dimension $n_a = n+a$. The extremal functions in the weighted Sobolev inequalities are fully characterized. Using a dimension reduction argument and  the weighted Sobolev inequalities, we can reproduce a subfamily of the sharp GN inequalities on the Euclidean space due to Del Pino and Dolbeault, and  obtain some new sharp  GN inequalities as well.  Our weighted inequalities are also extended to the domain $\R^{n-m}\times \R^m_+$ and the weights are $\omega(x,t) = t_1^{a_1}\dots t_m^{a_m}$, where $n\geq m$, $m\geq 0$ and $a_1,\cdots,a_m\geq 0$. A weighted $L^p$-logarithmic Sobolev inequality is derived from these inequalities.  
\end{abstract}

\section{Introduction}
In the book \cite{BGL}, Bakry, Gentil and Ledoux prove that for any $n\in \N^\ast$ and $a \geq 0$ such that $n+a > 2$, there exists a constant $S(n,a)$ such that for any smooth, compactly supported function $f$ on $\R^{n-1}\times \R_+\subset \R^n$, we have  
\begin{equation}\label{eq:BGLSobolevonhalfspace}
\left(\int_{\R^{n-1}}\int_0^\infty |f(x)|^{2^*}x_n^a\, dx\right)^{\frac1{2^*}}\leq S(n,a) \left(\int_{\R^{n-1}}\int_0^\infty |\nabla f(x)|^2 x_n^a\, dx\right)^{\frac12},
\end{equation}
with $ 2^* = \frac{2(n+a)}{n+a-2}$ . The best constant $S(n,a)$ is given by
$$S(n,a) =\left(\frac1{\pi (n+a)(n+a-2)}\right)^{\frac12}\left[\frac{2\, \pi^{\frac{1+a}{2}}}{\Gamma(\frac{1+a}2)}\, \frac{\Gamma(n+a)}{\Gamma(\frac{n+a}{2})}\right]^{\frac{1}{n+a}},$$
where $\Gamma$ is the usual \emph{Gamma function\text} defined by
$\Gamma(r) = \int_0^\infty t^{r-1}\,e^{-t} dt$ for $r>0$.

The Bakry, Gentil, Ledoux's proof of~\eqref{eq:BGLSobolevonhalfspace} is based on the \emph{Curvature-Dimension\text} condition. First, these authors use the stereographic projection to transport $\R^{n-1}\times \R_+$ equipped with the measure having density $x_n^a$ with respect to Lebesgue measure on $\R^{n-1}\times \R_+$ to a half sphere $\{x\in S^n\, \big|\, x_n > 0\}$ equipped with the measure having density $x_n^a$ with respect to the Riemannian measure of the sphere. The operator associated with this measure is $\Delta_{S^n} + a\, \nabla \left(\log v\right)$, where $v(x) =x_n$. This operator satisfies the \emph{Curvature-Dimension\text} condition $CD(n+a-1,n+a)$ on the half sphere $\{v > 0\}$. This condition implies a sharp Sobolev inequality on $\{ v > 0\}$ which is equivalent to~\eqref{eq:BGLSobolevonhalfspace} via the stereographic projection. 

By an argument of dimension reduction reproduced in~\cite{BGL} one can derive from the weighted inequality~\eqref{eq:BGLSobolevonhalfspace} in dimension $n+1$ the following sub-family  of the sharp GN inequalities on $\R^n$ due to Del Pino and Dolbeault \cite{DD1}: Given $p > 1$, then
\begin{equation}\label{eq:DelpinoDolbeault}
\left(\int_{\R^n}|f(x)|^{2p} dx\right)^{\frac1{2p}}\leq G(n,p) \left(\int_{\R^n}|\nabla f(x)|^2 dx\right)^{\frac\theta{2}} \left(\int_{\R^n}|f(x)|^{p+1}dx\right)^{\frac{1-\theta}{p+1}},
\end{equation}
where $\theta = \frac{n(p-1)}{p(n+2 - (n-2)p}$ is determined by scaling invariance. A version of~\eqref{eq:DelpinoDolbeault} with $0 < p<1$ can be found in \cite{DD1}. In \cite{DD2}, Del Pino and Dolbeault generalized these inequalities to $L_p-$norm of gradient, and obtained the following family of sharp GN inequalities: Given $1 < p < n$ and $\alpha \in (1,\frac{n}{n-p})$, then
\begin{equation}\label{eq:DDp}
\left(\int_{\R^n}|f|^{\alpha p}dx\right)^{\frac1{\alpha p}}\leq G(n,p,\alpha)\left(\int_{\R^n}|\nabla f|^p dx\right)^{\frac\theta p} \left(\int_{\R^n}|f|^{\alpha(p-1) +1} dx\right)^{\frac{1-\theta}{\alpha(p-1)+1}},
\end{equation}
$\theta$ is determined by scaling invariance. The proof of Del Pino and Dolbeault relies on standard techniques such as symmetrization, solution of one dimensional variational problem. A new and simple proof of~\eqref{eq:DDp} which uses the mass transportation method is given in \cite{CNV} by Cordero-Erausquin, Nazaret and Villani. This method does not require symmetrization and can be applied to any norm of gradients on $\R^n$ (Del Pino and Dolbeault's results are stated for the Euclidean norm only).

Our first motivation in the present work was to generalize the inequality~\eqref{eq:BGLSobolevonhalfspace} to $L_p$ norm of gradients. This required to seek for different approach, as the proof in~\cite{BGL} relied on the conformal structure of the $L^2$ inequality. 
We will see that a new, simple,  twist in the mass transportation technique can be used for our purposes. 
But before going on, we should emphasize first that the extension to $L^p$ norm of gradients of~\eqref{eq:BGLSobolevonhalfspace} has already been obtained in the very recent works of  Cabr\'e, Ros-Oton and Serra~\cite{CR, CRS}. We learned of these results while writing our paper, but actually  their approach is completely different (see the discussion below). The mass transport proof that we propose in the present work not only provides a simple approach to the weighted Sobolev inequalities in question, but it will also allow us to obtain some other, more general, inequalities. Nonetheless, for clarity, we choose to discuss the weighted Sobolev inequalities first, as it  shows clearly how the argument works.

We shall also mention that in his book (note page $562$ of \cite{Villani1}) Villani suggested that the method used to derive Sobolev type inequalities from the curvature-dimension criterion (Theorems $21.9$ and $21.12$) should also apply to the case of convex cones. Our proof below is more direct.

Let $||\cdot||$ be a norm on $\R^n$ and $B = \{x\in \R^n \,:\, ||x||\leq 1\}$ the associated unit ball. Its dual norm $||\cdot||_*$ is defined on $\R^n$ by
$$||x||_* = \sup_{||y||\leq 1}\, x\cdot y,$$
where $x\cdot y$ denotes the Euclidean scalar product of $x$ and $y$. Throughout the paper (unless otherwise stated), we denote 
$$\Om = \R^{n-1}\times \R_+$$
 the open half-space $\{x_n>0\}$ of $\R^n$, and $\om$ will stand for the positive weight 
$$\omega(x) =x_n^a,\quad\quad\forall\, x\in \Om.$$
Then $L^p(\Om,\om)$ will be  the space of all measurable functions $f$ such that
$$||f||_{L^p(\Om,\om)} =\left(\int_\Om |f(x)|^p\om(x)\, dx \right)^{\frac1p}<\infty.$$
For $p\geq 1$, we define $\dot{W}^{1,p}(\Om,\om)$ the space of all measurable functions $f$ such that its level sets $\{x\in\Om\, : \, |f(x)| >t\}, t > 0$, have finite measure with respect to measure of density $\om$ in $\Om$, and its distributional gradient belongs to $L^p(\Om,\om)$. For $f\in \dot{W}^{1,p}(\Om,\om)$, we denote
$$||\nabla f||_{L^p(\Omega,\omega)} = \left(\int_{\Omega} ||\nabla f(x)||_*^p\, \omega(x) dx\right)^{\frac1p}.$$

We next introduce a family of the functions $h_{p,a}$ in $\dot{W}^{1,p}(\Om,\om)$ with $1\leq p < n_a:=n+a$ by
\begin{equation}\label{eq:Sobextremalfunction}
\forall x\in \Omega, \qquad h_{p,a}(x) = 
\begin{cases}
\displaystyle (\sigma_{p,a} +||x||^q)^{-\frac{n_a-p}{p}}&\mbox{if $p > 1$  }\\ 
\frac{\ind_{B\cap \Om}(x)}{(\int_{B\cap \Om}x_n^a \,dx)^{\frac{n_a-1}{n_a}}}&\mbox{if $p =1$,}
\end{cases}
\end{equation}
where $q = p/(p-1)$, $B = \{x\in \R^n\, :\, ||x||\leq 1\}$ and 
$$\sigma_{p,a} =\left[\frac{\Gamma(\frac{n_a}{p})\Gamma(\frac{n_a}{q}+1)}{\Gamma(n_a)}\int_{B\cap\Om} x_n^a dx\right]^{\frac{p}{n_a}}$$ 
is normalized constant such that 
$$ \int_\Omega h_{p,a}(x)^{p^*}\omega(x)dx = 1,\quad p^* = \frac{n_a p}{n_a -p}.$$
We will see that these functions are extremal functions for the weighted Sobolev inequality in the next two propositions.
\begin{proposition}\label{maintheorem1}
Let $a\geq 0$. If $1 < p <n_a$, there exists a constant $S(n,a,p) $ such that for any  $f\in \dot{W}^{1,p}(\Omega,\omega)$, 
\begin{equation}\label{eq:swSobp}
||f||_{L^{p^*}(\Omega,\omega)}\leq S(n,a,p)\,||\nabla f||_{L^p(\Omega,\omega)}.
\end{equation}
The best constant $S(n,a,p)$ is given by
\begin{equation}\label{eq:bestconstant}
S(n,a,p) = \left(\frac{(p-1)^{p-1}}{n_a(n_a-p)^{p-1}}\right)^{\frac1{p}}\left[\frac{\Gamma(\frac{n_a}p)\Gamma(\frac{n_a(p-1)}{p}+1)}{\Gamma(n_a)} \int_{B\cap \Om} x_n^a dx\right]^{-\frac1{n_a}}
\end{equation}
Equality in~\eqref{eq:swSobp} holds if and only if
$$ f(x) = c\, h_{p,a}(\lambda(\x-\x_0,x_n)),$$
for some $c\in \R$, $\lambda >0$ and $\x_0\in \R^{n-1}$, and $h_{p,a}$ is given by~\eqref{eq:Sobextremalfunction}.
\end{proposition}

In the case $p = 2$ with the Euclidean norm $|\cdot|$, Proposition~\ref{maintheorem1} reproduces the inequality~\eqref{eq:BGLSobolevonhalfspace} of Bakry, Gentil and Ledoux. By adapting the dimension-reduction mentioned above to the case where $p\neq 2$, we will show in Section~\S 3 (see Theorem~\ref{FromSobtoGNS}), how to reproduce part of the family~\eqref{eq:DDp} of sharp GN inequalities on the Euclidean space (with an arbitrary norm). This gives a geometric justification to the family~\eqref{eq:DDp}.

Let us  mention that, interestingly enough,  even if we aim at the family~\eqref{eq:DDp} for the Euclidean norm only, we~\emph{need} to use, when $p\neq 2$, the Proposition~\ref{maintheorem1} for a \emph{non-Euclidean} norm. This gives more evidence that the extension to arbitrary norms in Sobolev type inequalities is not a purely formal matter, but it is indeed natural and useful. Let us also point out that we can actually extend the family~\eqref{eq:DDp} to the case $p>n$. These sharp GN inequalities on $\R^n$ for $p>n$ and $n\ge 2$ seem to be new.

\medskip

The case when $p=1$ of the weighted Sobolev inequality above is related to the functions of $\om-$bounded variation. A function $f:\Om\to \R$ is said to have $\om-$bounded variation if
$$\sup\left\{\int_\Om f\d(g\om)\, dx\, :\, g\in C_0^1(\Om,\R^n),\, \, ||g(x)||\leq 1, \,\,\, \forall\, x\in \Om \right\} < \infty.$$
If we denote $\BV_\om(\Om)$ the set of such functions, then $\dot{W}^{1,1}(\Om,\om)\subset BV_\om(\Om)$. Note that if $f$ has $\om-$bounded variation, then it has locally bounded variation on $\Om$. Let $Df$ be the vector valued Radon measure on $\Om$ such that
$$\int_\Om f \d\,\varphi \,dx = -\int_\Om \varphi\cdot Df $$  
for all $\varphi \in C_0^1(\Om,\R^n)$, and let $|Df|$ be its total variation. It follows from \cite[Theorem $1$, Chapter $5$]{EG} that there exists a $|Df|-$measurable function $\nu_f: \Om\to \R^n$ such that $|\nu_f|=1$ $|Df|-$a.e., and $d(Df)(x) = \nu_f(x)\, d(|Df|)(x)$.

It is readily to check that
$$\sup\left\{\int_\Om f\d(g\om)\, dx\, \big|\, g\in C_0^1(\Om,\R^n),\,\, ||g(x)||\leq 1, \,\, \forall\, x\in \Om \right\} = \int_\Om \om\, ||\nu_f||_*\, d(|Df|),$$
and if $f\in \dot{W}^{1,1}(\Om,\om)$, we have
$$||\nabla f||_{L^1(\Om,\om)} = \int_\Om \om\, ||\nu_f||_*\, d(|Df|).$$

A measurable set $E\subset \R^n$ has $\om-$bounded variation if $\ind_E \in \BV_\om(\Om)$, and we define its weighted perimeter with respect to $\Om$ and $||\cdot||$, denote $\nu_E = \nu_{\ind_E}$, by
$$P_{\om,||\cdot||}(E,\Om) = \int_\Om \om\,||\nu_E||_*\, d(|D\ind_E|).$$
If $E$ is Lipschitz, we have
$$P_{\om,||\cdot||}(E,\Om) = \int_{\partial E\cap \Om} ||n(x)||_*\, \om(x)\, d\,\mathcal{H}^{n-1}(x),$$
where $n(x)$ is the ($\mathcal{H}^{n-1}$-almost everywhere defined) unit outward normal at $x\in \partial E$ and $\mathcal{H}^{n-1}$ is the $n-1$ dimensional Hausdorff measure on $\partial E$. The unit ball $B = \{x\in \R^n\, :\, ||x||\leq 1\}$ has $\omega-$bounded variation, and its weighted perimeter is 
$$P_{\om,||\cdot||}(B,\Om) = n_a\int_{\Om\cap B} \om(x)\, dx,$$
since for $\mathcal{H}^{n-1}$-almost every $x\in \partial B$ we have $n(x) =x^*:=\nabla(||\cdot||)(x)$. The corresponding weighted $L_1-$Sobolev inequality is stated in the following proposition.
\begin{proposition}\label{mainL1}
There exists a constant $S(n,a,1)$ such that for any smooth compactly supported function $f$, then 
\begin{equation}\label{eq:swSob1}
||f||_{L^{\frac{n_a}{n_a -1}}(\Omega,\omega)}\leq S(n,a,1)\, ||\nabla f||_{L^1(\Omega,\omega)}.
\end{equation}
The best constant $S(n,a,1)$ is given by
\begin{equation}\label{eq:bestconstant1}
S(n,a,1) = n_a^{-1}\left(\int_{B\cap \Om} x_n^a dx\right)^{-\frac1{n_a}}.
\end{equation}
Inequality~\eqref{eq:swSob1} extends to all functions of $\om-$bounded variation, and equality holds if 
$$ f(x) = c\, h_{1,a}(\lambda(\x-\x_0,t)),$$
for some $c\in \R$, $\lambda > 0$, $\x_0\in \R^{n-1}$ and $h_{1,a}$ is given by~\eqref{eq:Sobextremalfunction}.
\end{proposition} 

By an approximation argument, we see that inequality~\eqref{eq:swSob1} is equivalent to a weighted isoperimetric inequality on $\Omega$ as follows,
 \begin{equation}\label{eq:Isopinequality}
\frac{P_{\om,||\cdot||}(E,\Om)}{\left(\int_{E\cap\Om}\om\,dx\right)^{\frac{n_a-1}{n_a}}}\, \geq \,\frac{P_{\om,||\cdot||}(B,\Om)}{\left(\int_{B\cap\Om}\om\,dx\right)^{\frac{n_a-1}{n_a}}},
\end{equation}
that is, among the set $E$ of $\om-$bounded variation with $\int_{E\cap\Om}\om\, dx = \int_{B\cap\Om}\om\, dx$, then $B$ has smallest weighted perimeter. Hence $B$ solves the isoperimetric problem on $\Om$ with weight $\om$. 

In the recent papers~\cite{CR, CRS}, Cabr\'e~\emph{et al.} stated the weighted isoperimetric inequality~\eqref{eq:Isopinequality} with a proof which relies on the so-called $ABP$ method for an  appropriate linear Neumann problem involving the following Laplacian-type operator
$$Lu = \om^{-1}\d\left(\om \nabla u\right) = \Delta u + \frac{\nabla\om\cdot \nabla u}{\om}.$$
Their method provides a proof for  isoperimetric inequalities with monominal weights (see \cite{CR}), which can be generalized to inequalities on open convex cone with general weights (see \cite{CRS}), which includes the half-space case discussed here (Proposition \ref{mainL1}). Then, they show that a weighted radial rearrangement argument of Talenti \cite{Tal1} allows to pass from the weighted isoperimetric inequality to the sharp weighted Sobolev inequalities, as the one stated in Proposition~\ref{maintheorem1}.

Our mass transport approach also allows to treat the more general case of weighted GN inequalities on $\Om$. Before stating the next result, let us introduce the functions $h_{\alpha,p,a}$ for $1< p < n_a$ and for $\alpha\in (0,\frac{n_a}{n_a-p})$, $\alpha \not= 1$, by
\begin{equation}\label{eq:ExtremalfunctionGN}
 h_{\alpha,p,a}(x) = (\sigma_{\alpha,p,a} +(\alpha -1)\,||x||^q)_+^{\frac{1}{1-\alpha}}, 
\end{equation}
where $q = p/(p-1)$, and  
\begin{equation*}\label{eq:GNextremalfunction}
\sigma_{\alpha,p,a} = 
\begin{cases}
\left[(\alpha -1)^{-\frac{n_a}{q}}\frac{\Gamma(\frac{\alpha p}{\alpha -1} -\frac{n_a}{q})\Gamma(\frac{n_a}{q}+1)}{\Gamma(\frac{\alpha p}{\alpha -1})}\int_{B\cap\Om} x_n^a\, dx\right]^{\frac{q(\alpha -1)}{\alpha pq -n_a(\alpha -1)}}&\mbox{ if $\alpha >1$  }\\ 
\left[(1 -\alpha)^{-\frac{n_a}{q}}\frac{\Gamma(\frac{\alpha p}{1-\alpha}+1)\Gamma(\frac{n_a}{q}+1)}{\Gamma(\frac{\alpha p}{1-\alpha}+\frac{n_a}{q}+1)}\int_{B\cap \Om} x_n^a\, dx\right]^{\frac{q(\alpha-1)}{\alpha pq -n_a(\alpha-1)}}&\mbox{ if $\alpha < 1$ }.
\end{cases}
\end{equation*}
 is normalized constant such that 
$$ \int_\Omega h_{\alpha,p,a}(x)^{\alpha p}\,\omega(x)\,dx = 1.$$
Note that for $\alpha <1$, the function $h_{\alpha,p,a}$ has compact support on $\Omega$, while  for $\alpha > 1$, it is positive everywhere, decaying polynomially at infinity.

The weighted GN inequalities on $\Om$ read as follows
\begin{theorem}\label{maintheorem2}
Let $p\in (1,n_a)$ and $\alpha\in (0,\frac{n_a}{n_a -p}]$, $\alpha \not= 1$.
\item (i) If $\alpha >1$, there exists a constant $G_{n,a}(\alpha,p)$ such that for any $f\in \dot{W}^{1,p}(\Omega,\omega)$,
\begin{equation}\label{eq:swGNineq}
||f||_{L^{\alpha p}(\Omega,\omega)} \leq G_{n,a}(\alpha,p)||\nabla f||_{L^p(\Omega,\omega)}^\theta ||f||_{L^{\alpha(p-1)+1}(\Omega,\omega)}^{1-\theta},
\end{equation} 
where
\begin{equation}\label{eq:theta}
 \theta =\frac{n_a(\alpha-1)}{\alpha(n_ap -(\alpha p+1 -\alpha)(n_a -p))} =\frac{p^*(\alpha -1)}{\alpha p(p^* -\alpha p+\alpha -1)}, 
\end{equation}
the best constant $G_{n,a}(\alpha,p)$ takes the explicit form, denoting $y =\frac{\alpha(p-1) +1}{\alpha -1}$,
\begin{equation}\label{eq:explicitconstant}
G_{n,a}(\alpha,p) = \left[\frac{y(\alpha-1)^p}{q^{p-1}n_a}\right]^{\frac{\theta}p}\left[\frac{qy - n_a}{qy}\right]^{\frac{1}{\alpha p}}\left[\frac{\Gamma(y)}{\Gamma(y-\frac{n_a}{q})\Gamma(\frac{n_a}{q}+1)\int_B x_n^a\, dx}\right]^{\frac{\theta}{n_a}}.
\end{equation}
\item (ii) If $\alpha < 1$, there exists a constant $N_{n,a}(\alpha,p)$ such that for any $f\in \dot{W}^{1,p}(\Omega,\omega)$,
\begin{equation}\label{eq:swGNineq1}
||f||_{L^{\alpha (p -1) +1}(\Omega,\omega)}\leq N_{n,a}(\alpha,p)||\nabla f||_{L^p(\Omega,\omega)}^\theta ||f||_{L^{\alpha p}(\Omega,\omega)}^{1-\theta},
\end{equation}
where
\begin{equation}\label{eq:theta1}
\theta =\frac{n_a(1-\alpha)}{(\alpha p+1-\alpha)(n-\alpha(n-p))} = \frac{p^*(1-\alpha)}{(p^*-\alpha p)(\alpha p+ 1-\alpha)},
\end{equation}
the best constant $N_{n,a}(\alpha,p)$ takes the explicit form, denoting $z = \frac{\alpha p -\alpha +1}{1 -\alpha }$,
\begin{equation}\label{eq:explicitconstant1}
N_{n,a}(\alpha,p) = \left[\frac{z(1-\alpha)^p}{q^{p-1}n_a}\right]^{\frac{\theta}p}\left[\frac{qz}{qz+n_a}\right]^{\frac{1-\theta}{\alpha p}}\left[\frac{\Gamma(z+1+\frac{n_a}{q})}{\Gamma(z+1)\Gamma(\frac{n_a}{q}+1)\int_B x_n^a \, dx}\right]^{\frac{\theta}{n_a}}.
\end{equation}

Moreover, equality in~\eqref{eq:swGNineq} and~\eqref{eq:swGNineq1} holds if 
$$ f(x) = c\, h_{\alpha,p,a}(\lambda(\x-\x_0,x_n)),$$
for some $c\in \R$, $\lambda > 0$, $\x_0\in \R^{n-1}$ and $h_{\alpha,p,a}$ is given by~\eqref{eq:ExtremalfunctionGN}.
\end{theorem}

More generally, our mass transportation proof can be used to generalize Propositions~\ref{maintheorem1} and~\ref{mainL1}, and Theorem~\ref{maintheorem2} to the domain
$$\Sigma =\R^{n-m}\times \R^{m}_+ = \{y\in \R^n\, : \,  y_j>0 \textrm{ for } j=n-m+1, \ldots, n\},$$
where $n \geq m\geq 0$, with the monominal weights
\begin{equation}\label{eq:generalweight}
\sigma(x,t) = t_1^{a_1}\cdots t_m^{a_m},
\end{equation}
where $ a_1,\cdots, a_m\geq 0$ and $(x,t)$ denotes vector $(x_1,\cdots,x_{n-m}, t_1,\cdots, t_m)\in \R^{n-m}\times \R^m_+$. The corresponding fractional dimension will then be
$$n_a = n+ a_1+\cdots+ a_m.$$
This generalization is given in section \S4. As before, the result for the particular case of the weighted Sobolev inequalities is already contained in work of Cabr\'e and Ros-Oton~\cite{CR}, with the very different proof mentioned above.

We know that, on $\R^n$, the sharp $L_p-$logarithmic Sobolev inequalities are the limit case of the GN inequalities~\eqref{eq:DDp} when $1\leq p <n$ and $\alpha\to 1$ (see \cite{DD1,DD2}). Following the idea of these papers, we can deduce from Theorem \ref{maintheorem2} a family of sharp weighted $L_p-$logarithmic Sobolev inequalities on $\Omega$ with weighted $\omega$, by letting $\alpha \to 1$ when $1\leq p < n_a$. By using an idea of Beckner and Pearson \cite{BP} and the general form of Proposition~\ref{maintheorem1} on $\R^{n-m}\times \R^m_+$, we can extend the family of sharp weighted $L_p-$logarithmic Sobolev inequalities to all $p\geq 1$. Hence we get:
\begin{proposition}\label{LplogSobineq}
Let $p\geq 1$, then for any $f\in W^{1,p}(\Omega,\omega)$ such that $\int_\Omega |f|^p\,\omega \,dx = 1$, then 
\begin{equation}\label{eq:LplogSob}
\int_\Omega |f(x)|^p\ln(|f(x)|^p)\,\omega(x) dx \leq \frac{n_a}{p}\ln\left[\mathcal{L}_{n,a}(p)\int_\Omega ||\nabla f||_*^p\,\omega(x) dx\right],
\end{equation}
where
\begin{equation}\label{eq:LplogSobconstant}
\mathcal{L}_{n,a}(p) = 
\begin{cases}
\frac{p}{n_a}\left(\frac{p-1}{e}\right)^{p-1}\left[\Gamma(\frac{n_a}{q}+1)\int_B x_n^a\, dx\right]^{-\frac{p}{n_a}}&\mbox{if $p> 1$}\\
\frac{1}{n_a}\left[\int_B x_n^a \, dx\right]^{-\frac{1}{n_a}}&\mbox{if $p =1$}.
\end{cases}
\end{equation}
Equality in~\eqref{eq:LplogSob} holds if 
$$f(x) = b e^{-a||x -(\x_0,0)||^q},$$
for some $a > 0$, $\x_0\in \R^{n-1}$ and $|b|^{-p} = \int_{\Om}e^{-pa||x||^q}\, \omega(x) dx$ when $p > 1$, and
$$f(x) = b \ind_B(a(\x-\x_0,x_n)), $$
for some $a > 0$ and $\x_0\in \R^{n-1}$ and $|b| = a^{\frac{n_a}p}\left(\int_B x_n^a\, dx\right)^{-\frac1p}$ when $p=1$.
\end{proposition}

On $\R^n$, sharp $L_p-$logarithmic Sobolev inequalities for $1\leq p < n$ were proved first by Del Pino and Dolbeault \cite{DD1, DD2} by considering the above mentioned limit of the sharp GN inequalities. But their results concerned only the Euclidean norm. Next, Gentil \cite{Gen} extended the result of Del Pino and Dolbeault for all $p\geq 1$ and for any norm on $\R^n$ by using the Pr\'{e}kopa-Leindler inequality and a special Hamilton-Jacobi equation. The case $p=2$ is interesting since it is equivalent to the Gross's logarithmic Sobolev inequality for the Gaussian measure. The case $p=1$ was proved by Beckner \cite{Bec}. Another proof of the sharp $L_p-$logarithmic Sobolev inequality for all $p\geq 1$ is given in \cite{AGK} and in \cite{Cor} where the authors exploit the mass transportation method. It is clear that this method gives us, using the techniques of the present paper, a direct proof of the inequality~\eqref{eq:LplogSob}, but we think it is also interesting to see them as a consequence of the sharp weighted Sobolev inequalities on the set $\Sigma$ with the weight $\sigma$. The proof of the Proposition~\ref{LplogSobineq} is given in section \S4.

Our method to prove the Propositions~\ref{maintheorem1} and ~\ref{mainL1}, and the Theorem~\ref{maintheorem2} is inspired by the work of Cordero-Erausquin, Nazaret and Villani \cite{CNV} which is based on the mass transportation method. We only use the so-called \emph{Brenier map\text}, the arithmetic-geometric inequality and H\"older's inequality (or Young's inequality). In recent years, mass transportation method has been used succesfully to prove some sharp inequalities in Analysis and Geometry, for example, see \cite{AGK,Bar,Cor,CNV,Gar,MV,MV2,Naz,Villani}.  Let us recall the results of Brenier and McCann.

Let $\mu$ and $\nu$ be two Borel probability measures on $\R^n$ such that $\mu$ is absolutely continuous with respect to Lebesgue measure. Then there exists a convex function $\varphi$ such that
\begin{equation}\label{eq:push-forward}
 \int b(y)d\nu(y) = \int b(\nabla\varphi(x))d\mu(x) .
\end{equation}
for every bounded or positive, Borel measurable function $b:\R^n\to \R$ (see \cite{Bre, McCann}). Furthermore, $\overline{\nabla\varphi(\supp \mu)} = \supp \nu$ and $\nabla \varphi$ is uniquely determined $d\mu-$almost everywhere. We call $\nabla\varphi$ the \emph{Brenier map} which transports $\mu$ to $\nu$. See \cite{Villani} for a review and dicussion of existing proofs of this map. Since $\varphi$ is convex, it is differentiable almost everywhere on its domain $\{ \varphi < \infty\}$; in particular, it is differentiable $d\mu-$almost everywhere. 

Let $\mu$ and $\nu$ be absolutely continuous with respect to the Lebesgue measure, with densities $F$ and $G$ respectively and let $\nabla\varphi$ be the Brenier map transporting $\mu$ onto $\nu$. Then we have
\begin{equation}\label{eq:push-forward1}
\int b(y)G(y)dy = \int b(\nabla\varphi(x))F(x)dx
\end{equation}
for every bounded or positive, Borel measurable function $b:\R^n\to \R$. If $\varphi$ is of class $C^2$, the change of variables $y =\nabla\varphi(x)$ in~\eqref{eq:push-forward1} shows that $\varphi$ solves the \emph{Monge-Amp\`{e}re equation \text}
\begin{equation}\label{eq:MAequation}
F(x) = G(\nabla\varphi(x))\det{D^2\varphi(x)}.
\end{equation}
Here $D^2\varphi(x)$ stands for the Hessian matrix of $\varphi$ at point $x$. Unfortunately, $\varphi$ is not $C^2$ in general. However,~\eqref{eq:MAequation} always holds in the $F(x)dx-$almost everywhere sense (see \cite{McCann1}) and $D^2\varphi$ should be understood in Aleksandrov sense, i.e, as the absolutely continuous part of the distributional Hessian of the convex function $\varphi$; we denote it by $D^2_A\varphi$. Moreover, this equation holds without further assumption on $F$ and $G$ beyond integrability. Then 
\begin{equation}\label{eq:MAequation1}
F(x) = G(\nabla\varphi(x))\det{D^2_A\varphi(x)}
\end{equation}
$F(x)dx-$almost everywhere.

As mentioned above, our proofs require to use H\"older's and Young's inequality. We recall them below. Let $||\cdot||$ be a norm on $\R^n$, $B$ its unit ball and $||\cdot||_*$ its dual norm. Then for any $\lambda >0$, Young's inequality holds
\begin{equation}\label{eq:Youngineq}
X\cdot Y \leq \frac{\lambda^{-p}}p ||X||_*^p + \frac{\lambda^q}q ||Y||^q, \quad\, q = \frac{p}{p-1}.
\end{equation}
For $X:\R^n\to (\R^n,||\cdot||_*)$ in $L^p$ and $Y: \R^n\to (\R^n,||\cdot||)$ in $L^q$, integration of~\eqref{eq:Youngineq} and optimization in $\lambda$ gives H\"older's inequality in the form
\begin{equation}\label{eq:Holderineq}
\int_{\R^n} X(x) \cdot Y(x)\, dx \leq \left(\int_{\R^n}||X(x)||_*^p \, dx\right)^{\frac1p}\, \left(\int_{\R^n} ||Y(x)||^q\, dy\right)^{\frac1q}.
\end{equation}

Since $||\cdot||$ is a Lipschitz function on $\R^n$, it is differentiable almost everywhere. When $0\not=x\in \R^n$ is a point of differentiability, the gradient of the norm at $x$ is the unique vector $x^* = \nabla(||\cdot||)(x)$ such that
\begin{equation}\label{eq:gradientofnorm}
||x^*||_* = 1, \quad\quad x\cdot x^* = ||x|| = \sup_{||y||_* = 1}x\cdot y.
\end{equation}
These equalities will be used to verify the extremality of $h_{p,a}$ and $h_{\alpha,p,a}$.

The rest of this paper is organized as follows: Section \S 2 is devoted to the proof of Proposotions~\ref{maintheorem1} and~\ref{mainL1}, and Theorem~\ref{maintheorem2}; actually we will prove a bit more, namely a duality principle, as in~\cite{CNV}. In the section \S 3, we show how to get (known and new) sharp GN inequalities on the Euclidean space from Proposition~\ref{maintheorem1}. The last section \S 4 is devoted to the extension of the results  to the domain $\Sigma$ with the monomial weight $\sigma$ given by~\eqref{eq:generalweight} and to the proof of Proposition~\ref{LplogSobineq} above.

\section{Proof of Propostions~\ref{maintheorem1} and~\ref{mainL1}, and Theorem \ref{maintheorem2}}

In this section, we give a mass transportation proof of Propositions~\ref{maintheorem1} and~\ref{mainL1}, and Theorem~\ref{maintheorem2}. We will need the following central lemma,
\begin{lemma}\label{technical}
Let $a\geq 0, n_a = n+a$ and let $1\not=\gamma \geq 1 -\frac1{n_a}$. Let $F$ and $G$ be two nonegative functions on $\Om = \R^{n-1}\times \R_+ \subset \R^n$ with $\int_\Omega F\,\omega\, dx =\int_\Omega G\, \omega\, dx =1$, and such that $F^\gamma$ is $C^1$ on $\Om$ and $F,G$ are compactly supported on $\overline{\Om}$. Then if  $\nabla\varphi$ is the Brenier map pushing $F\,\omega\, dx$ forward to $G\,\omega\, dx$, we have 
\begin{equation}\label{eq:integraionbypart}
\frac{1}{1-\gamma}\int_\Om G^\gamma \, \omega \, dx \leq \frac{1-n_a(1-\gamma)}{1-\gamma}\int_\Om F^\gamma\, \om\, dx -\int_\Om \nabla F^\gamma\cdot \nabla\varphi\, \om\, dx.
\end{equation}
\end{lemma}
\begin{proof}
Since $F$ and $G$ are compactly supported on $\overline{\Om}$, then $\nabla\varphi$ is bounded and $\varphi$ is Lipschitz (on the support of $F$). It follows from~\eqref{eq:MAequation1} that
$$F(x)x_n^a = G(\nabla\varphi(x))\left(\partial_n\varphi(x)\right)^{a}\det D_A^2\varphi(x),$$
with the notation $\partial_n = \frac{\partial}{\partial x_n}$. Then, for $1 -\frac{1}{n_a}\leq \gamma \not=1$, we have
$$ G(\nabla\varphi(x))^{\gamma-1} = F(x)^{\gamma-1}\left(\frac{\partial_n\varphi(x)}{x_n}\right)^{a(1-\gamma)}\left(\det D^2_A\varphi(x,t)\right)^{1-\gamma}. $$
We next claim that
\begin{align}\label{eq:basicinequality}
\frac{1}{1-\gamma}\,G(\nabla\varphi(x))^{\gamma-1}\leq F(x)^{\gamma -1}\,\frac{1-n_a(1-\gamma)}{1-\gamma} + F(x)^{\gamma-1}L_A\varphi(x),
\end{align}
where we denote
$$ L_A\varphi(x) = a\,\frac{\partial_n\varphi(x)}{x_n} +\Delta_A\varphi(x). $$
It is worth noting that the operator $L = \Delta + a\, x_n^{-1}\, \partial_n$ is the Laplacian associated to the weight $\om(x) = x_n^a$. To prove \eqref{eq:basicinequality}, it suffices to show that
\begin{equation}\label{eq:them}
\frac{1}{1-\gamma} A^{a(1-\gamma)} \lt(\det M\rt)^{1-\gamma} \leq \frac{1-n_a(1-\gamma)}{1-\gamma} + aA +\tr(M).
\end{equation}
for any $A > 0$ and any $n\times n$ nonnegative, symmetric matrix $M$. We divide the proof of \eqref{eq:them} in two cases corresponding to $\gamma < 1$ and to $\gamma > 1$.

If $1-\frac1{n_a}\leq \gamma < 1$, we have
\begin{align*}
\frac{1}{1-\gamma} A^{a(1-\gamma)} \lt(\det M\rt)^{1-\gamma}& = \frac1{1-\gamma} A^{a(1-\gamma)}\lt((\det M)^{\frac1n}\rt)^{n(1-\gamma)} 1^{1-n_a(1-\gamma)}\\
&\leq \frac1{1-\gamma}\lt(a(1-\gamma)A + n(1-\gamma)(\det M)^{\frac1n} + (1-n_a(1-\gamma))1 \rt)\\
&\leq \frac{1-n_a(1-\gamma)}{1-\gamma} + aA +\tr(M),
\end{align*}
where we use the elementary inequalities 
$$x^{\al_1}y^{\al_2}z^{\al_3} \leq \al_1x + \al_2 +\al_3z$$
for $x,y,z, \al_1,\al_2,\al_3 \geq 0$ and $\al_1 +\al_2+\al_3 =1$ in the first estimate, and  
$$(\det M)^{\frac1n} \leq \frac1n \tr(M)$$
in the second estimate.

If $\gamma > 1$, then \eqref{eq:them} is equivalent to the following inequality
\begin{equation}\label{eq:equiv}
\frac1{1+n_a(\gamma-1)}A^{a(1-\gamma)} \lt(\det M\rt)^{1-\gamma} + \frac{a(\gamma-1)}{1+n_a(\gamma-1)}A + \frac{\gamma-1}{1+n_a(\gamma-1)}\tr(M)\geq 1.
\end{equation}
Using two elementary inequalities mentioned above, we have
\begin{align*}
&\frac1{1+n_a(\gamma-1)}A^{a(1-\gamma)} \lt(\det M\rt)^{1-\gamma} + \frac{a(\gamma-1)}{1+n_a(\gamma-1)}A + \frac{\gamma-1}{1+n_a(\gamma-1)}\tr(M)\\
&\geq \frac1{1+n_a(\gamma-1)}A^{a(1-\gamma)} \lt(\det M\rt)^{1-\gamma} + \frac{a(\gamma-1)}{1+n_a(\gamma-1)}A + \frac{n(\gamma-1)}{1+n_a(\gamma-1)}(\det M)^{\frac1n}\\
&\geq \lt(A^{a(1-\gamma)} \lt(\det M\rt)^{1-\gamma}\rt)^{\frac1{1+n_a(\gamma-1)}}\, A^{\frac{a(\gamma-1)}{1+n_a(\gamma-1)}}\, \lt((\det M)^{\frac1n}\rt)^{\frac{n(\ga -1)}{1+ n_a(\ga -1)}}\\
&= 1,
\end{align*}
which is exactly \eqref{eq:equiv}.

Multiplying  both sides of~\eqref{eq:basicinequality} with $F(x)\omega(x)$, integrating the obtained inequality on $\Omega$, and using the definition of mass transport~\eqref{eq:push-forward1}, we get
\begin{equation}\label{eq:important}
\frac{1}{1-\gamma}\int_\Omega G^{\gamma}\, \omega\, dx \leq \frac{1-n_a(1-\gamma)}{1-\gamma}\int_\Omega F^\gamma\,\omega\, dx + \int_\Omega F^\gamma\, L_A\varphi \,\omega\, dx.
\end{equation}
To treat the integration by parts, let $\theta: [0,\infty) \to [0,1]$ be such that $\theta$ is $C^\infty$, increasing, $\theta \equiv 0 $ on $[0,1]$, and $\theta \equiv 1 $ on $[2,\infty)$. Set $\theta_k(x) = \theta(k x_n)$ for $x\in \Om$, and $F_k = F^\gamma\theta_k$ with $ k\geq 1$. Since $L_A\varphi \geq 0$, by Fatou's lemma
$$\int_\Om F^\gamma(x)\, L_A\varphi(x)\, \om(x)\, dx \leq \liminf_{k\to\infty}\int_\Om F_k(x)\, L_A\varphi (x)\, \om(x)\, dx.$$
We have $F_k\in C_0^1(\Omega)$, so there exists a sequence of nonegative functions $\{F_{k,m}\}_m\in C_0^\infty(\Om)$ such that
$$\lim_{m\to\infty} F_{k,m} = F_k\quad\, \text{ in }W^{1,1}(\Omega),$$
and $\lim_{m\to\infty}F_{k,m}(x) =F_k(x)$ for almost everywhere $x\in \Om$. Since $\Delta_A\varphi \leq \Delta_{\mathcal{D}'}\varphi$, where $\Delta_{\mathcal{D}'}\varphi$ stands for distributional Laplacian of $\varphi$, and by integration by parts, we get
$$\int_\Om F_{k,m}(x)\, L_A\varphi (x)\, \om(x)\, dx \leq -\int_\Om \nabla(F_{k,m})\cdot \nabla\varphi\, \om\, dx.$$
Since $L_A\varphi \geq 0$, let $m\to \infty$, by Fatou's lemma and the boundedness of $\nabla\varphi$, we get
$$\int_\Om F_k(x)\, L_A\varphi(x)\, \om(x)\, dx \leq -\int_\Om \nabla F_k\cdot\nabla\varphi\, \om\, dx.$$
Next, note that
\begin{align*}
-\int_\Om \nabla F_k\cdot\nabla\varphi\, \om\, dx&= -\int_\Om \theta_k \nabla F^\gamma\cdot \nabla\varphi\, \om\, dx - k\int_\Om F(x)^\gamma \theta'(kx_n)\partial_n\varphi(x) \om(x) dx\\
&\leq -\int_\Om \theta_k \nabla F^\gamma\cdot \nabla\varphi\, \om\, dx,
\end{align*}
since $\theta$ is increasing and $\partial_n\varphi(x)\geq 0$ because $\nabla\varphi(x) \in \Om$ for $x$ on the support of $F$. Letting $k\to\infty$, we get
$$\int_\Om F^\gamma\, L_A\varphi\, \om\, dx \leq -\int_\Om \nabla F^\gamma\cdot \nabla\varphi\, \om\, dx,$$
which, together with~\eqref{eq:important}, gives~\eqref{eq:integraionbypart}.
\end{proof}

With Lemma \ref{technical} on hand, we are now ready to prove the Propositions~\ref{maintheorem1} and ~\ref{mainL1}, and the Theorem\ref{maintheorem2}.

\emph{Proof of Proposition \ref{maintheorem1}\text}: Proposition~\ref{maintheorem1} follows from the following two propositions. The first one states a duality principle which is of independent interest.

\begin{proposition}\label{plarge1}
Let $1< p < n_a$ and $q = p/(p-1)$. For any function $f \in \dot{W}^{1,p}(\Omega,\omega)$ and $g\in L^{p^*}(\Om,\om)$ with $||f||_{L^{p^*}(\Omega,\omega)} = ||g||_{L^{p^*}(\Omega,\omega)}$, then
\begin{equation}\label{eq:Sobcompact1}
\frac{\int_\Omega |g(x)|^{p^*(1-\frac{1}{n_a})}\,\omega(x)\,dx}{\left(\int_\Omega ||y||^q\,|g(y)|^{p^*}\,\omega(y)\,dy\right)^{\frac{1}{q}}}\leq \frac{p(n_a-1)}{n_a(n_a-p)}||\nabla f||_{L^p(\Omega,\omega)},
\end{equation}
with equality if $f = g = h_{p,a}$.

As immediate consequences we have
\item(i) The duality principle
\begin{equation}\label{eq:dualprinciple}
\sup_{||g||_{L^{p^*}(\Omega,\omega)}=1}\frac{\int_\Omega |g(x)|^{p^*(1-\frac{1}{n_a})}\,\omega(x)\,dx}{\left(\int_\Omega ||y||^q\,|g(y)|^{p^*}\,\omega(y)\,dy\right)^{\frac{1}{q}}} = \frac{p(n_a-1)}{n_a(n_a-p)}\, \inf_{||f||_{L^{p^*}(\Omega,\omega)}=1}||\nabla f||_{L^p(\Omega,\omega)},
\end{equation}
with $h_{p,a}$ is extremal in both variational problems.

\item(ii) The sharp weighted Sobolev inequality: If $0\not= f\in \dot{W}^{1,p}(\Om,\om)$, then
\begin{equation}\label{eq:Sob}
\frac{||\nabla f||_{L^p(\Om,\om)}}{||f||_{L^{p^*}(\Om,\om)}}\geq ||\nabla h_{p,a}||_{L^p(\Om,\om)}.
\end{equation}
\end{proposition}

\begin{proof}

By the homogeneity, we can assume $||f||_{L^{p^*}(\Omega,\omega)}=||g||_{L^{p^*}(\Omega,\omega)} =1$. It is well-known that if $f\in \dot{W}^{1,p}(\Om,\om)$ then $||\nabla |f|\,||_* \leq ||\nabla f||_*$. Thus, without loss of generality, we may assume that $f$ and $g$ are nonegative. By standard approximation, we can assume that $f$ and $g$ are compactly supported on $\overline{\Om}$ and $f$ is in $C^1(\overline{\Om})$. Applying Lemma \ref{technical} to $F = f^{p^*}$, $G = g^{p^*}$ and $\gamma = 1 -\frac1{n_a}$, we get
\begin{equation}\label{eq:dualprinciple1}
\int_\Om g^{p^*(1-\frac1{n_a})}\, \om\, dx \leq -\frac{(n_a-1)p}{n_a(n_a-p)}\int_\Om \nabla f\cdot f^{\frac{p^*}q}\nabla\varphi\, \om\, dx.
\end{equation}
Applying H\"older's inequality~\eqref{eq:Holderineq}, we obtain
\begin{equation}\label{eq:dualprinciple2}
\int_\Om g^{p^*(1-\frac1{n_a})}\, \om\, dx \leq \frac{(n_a-1)p}{n_a(n_a-p)}||\nabla f||_{L^p(\Om,\om)}\left(\int_\Om f^{p^*}||\nabla\varphi||^q\, \om\, dx\right)^{\frac1q}.
\end{equation}
Inequality~\eqref{eq:Sobcompact1} follows from~\eqref{eq:dualprinciple2} and the definition of mass transportation.

If $f = g = h_{p,a}$, we must have $\nabla\varphi(x) =x$ for all $x\in \Om$, hence $D^2_A\varphi =D^2\varphi = \id_n$. This implies immediately that $\Delta_A\varphi$ becomes classical Laplacian, and  equality in~\eqref{eq:basicinequality} holds for all $1-\frac1{n_a} \leq \gamma\not=1$.  We can then use a simple integration by parts to get
$$\int_\Om h_{p,a}^{\frac{(n_a-1)p}{n_a-p}} L_A\varphi\, \om\, dx = -\frac{(n_a-1)p}{n_a-p}\int_\Om \nabla f\cdot f^{\frac{p^*}q}\nabla\varphi\, \om\, dx,$$
since in this case $\partial_n\varphi = 0$ on $\partial \Om$. Therefore,~\eqref{eq:dualprinciple1} becomes an equality. Moreover, if $f=g =h_{p,a}$ then
$$\nabla f(x) = -\frac{n_a-p}{p-1}\left(\sigma_{p,a}+||x||^q\right)^{-\frac{n_a}p} ||x||^{q-1}x^*$$
almost everywhere. This ensures the equality in H\"older's inequality~\eqref{eq:Holderineq}. So there is equality in~\eqref{eq:dualprinciple2} (and then in~\eqref{eq:Sobcompact1}) if $f= g=h_{p,a}$. Proposition \ref{plarge1} is proved. 
\end{proof}
The case of equality is treated in the following proposition:
\begin{proposition}\label{equality}
A function $f\in\dot{W}^{1,p}(\Om,\om)$ is optimal in the weighted Sobolev inequality~\eqref{eq:swSobp} if and only if there exist $c\in\R$, $\lambda \not=0$ and $\x_0\in \R^{n-1}$ such that
$$f(x) = c\, h_{p,a}(\lambda(\x-\x_0,x_n)).$$
\end{proposition}
\begin{proof}
We follow the idea of Cordero-Erausquin, Nazaret and Villani, which amounts to trace back equality cases in the mass transport proof. For this, we will show that the intermediate step~\eqref{eq:IbP2} below is valid in general (without regularity assumptions).

First, note that it is enough to prove this proposition for nonnegative function $f$ and $\int_\Om f^{p^*}\om = 1$. Let $\nabla\varphi$ be the Brenier map pushing $f(x)^{p^*}\om(x)\, dx$ forward to $h_{p,a}(x)^{p^*}\om(x)\,dx$. We have $f^{\frac{p^*}{q}}\nabla\varphi\in L^q(\Om,\om)$ since $h_{p,a}(x)^{\frac{p^*}{q}}x \in L^q(\Om,\om)$. 

Denote $O$ be the interior of the set $\{x: \varphi(x) < \infty\}$, we know that support of $f$ is contained in $\overline{O}$. Fix $x_0=(\x_0,t_0)\in O\cap\Om$. Choose $k$ such that $\frac1k < t_0$ and denote $f_k = f\theta_k$, where $\theta_k(x) = \theta(kx_n)$ is as in the proof of Lemma \ref{technical}. The support of $f_k$ is contained in $\{x_n \geq \frac1k\}$ and  $f_k\in \dot{W}^{1,p}(\R^n)$. For $\epsilon > 0$ small enough ($\epsilon\ll t_0 -\frac1k$), we define
$$f_{k,\epsilon} = \min\left\{f_k\left(x_0 + \frac{x-x_0}{1-\epsilon}\right), f_k(x) \chi(\epsilon x)\right\}$$
where $\chi$ is a $C^\infty$ cut-off function with $0\leq \chi \leq 1$, $\chi(x) = 1$ for $|x|\leq \frac12$, $\chi(x) = 0$ for $|x|\geq 1$. Then $f_{k,\epsilon}\in \dot{W}^{1,p}(\Om,\om)$. For $\delta >0$, denote $f_{k,\epsilon,\delta} = f_{k,\epsilon}\star \psi_\delta$, where $\psi_\delta = \delta^{-n}\psi(\frac{\cdot}{\delta})$ and $\psi\in C^\infty_0(\R^n)$ is a positive function, $\int_{\R^n} \psi =1$. For $\delta $ small enough ($\delta$ is smaller than the distance from support of $f_{k,\epsilon}$ to $\partial O$),  $f_{k,\epsilon,\delta}$ is compactly supported in $O$ and smooth, i.e.. $f_{k,\epsilon,\delta}\in C^\infty_0(O)$, and, since $L_A\varphi \leq L_{\mathcal{D}'}\varphi$, we have
\begin{equation}\label{eq:IbP}
\int (f_{k,\epsilon,\delta})^{\frac{(n_a-1)p}{n_a-p}}L_A\varphi\, \om \leq -\frac{(n_a-1)p}{n_a-p}\int (f_{k,\epsilon,\delta})^{\frac{p^*}{q}}\nabla f_{k,\epsilon,\delta}\cdot \nabla\varphi\, \om.
\end{equation}
Since the support of $f_{k,\epsilon}$ is contained in the set $\{x_n \geq \epsilon t_0 + \frac{1-\epsilon}k\}$ for any $k,\epsilon$, we have $f_{k,\epsilon,\delta}\to f_{k,\epsilon}$ when $\delta \to 0$.

By using the argument of Cordero-Erausquin, Nazaret and Villani \cite[Proof of Lemma 7]{CNV}, we can let $\delta\to 0$ and then let $\epsilon \to 0$ in~\eqref{eq:IbP} to get
\begin{align}\label{eq:IbP1}
\int_\Om f_k^{\frac{(n_a-1)p}{n_a-p}} L_A\varphi\, \om &\leq -\frac{(n_a-1)p}{n_a-p} \int_\Om f_k^{\frac{p^*}{q}} \nabla f_k\cdot \nabla\varphi\, \om\notag\\
& =-\frac{(n_a-1)p}{n_a-p}\int_\Om f_k^{\frac{p^*}{q}}\theta_k \nabla f\cdot \nabla\varphi \,\om \notag\\
&\, \, \, \, -\frac{(n_a-1)p}{n_a-p}\int_\Om f^{\frac{(n_a-1)p}{n_a-p}}\,\theta(kx_n)^{\frac{p^*}{q}}\,k\,\theta'(kx_n)\,\partial_n\varphi \,\om.
\end{align}
Since $\theta$ is increasing, we get, by letting $k\to\infty$,
\begin{equation}\label{eq:IbP2}
\int_\Om f^{\frac{(n_a-1)p}{n_a-p}} L_A\varphi\, \om\leq -\frac{(n_a-1)p}{n_a-p}\int_\Om f^{\frac{p^*}{q}} \nabla f\cdot \nabla\varphi \,\om.
\end{equation}
Applying inequality~\eqref{eq:important} to $f^{p^*}$, $h_{p,a}^{p^*}$ and $\gamma =1-\frac1{n_a}$ (this inequality always holds without further assumption of smoothness for $F$ or of compact support for $G$), we obtain
\begin{equation}\label{eq:IbP3}
\int_\Om h_{p,a}^{\frac{(n_a-1)p}{n_a-p}} \, \om \leq -\frac{(n_a-1)p}{n_a(n_a-p)}\int_\Om f^{\frac{p^*}{q}} \nabla f\cdot \nabla\varphi \,\om.
\end{equation}
By using H\"older's inequality, we get the sharp weighted Sobolev inequality~\eqref{eq:swSobp}. Since $f$ is an optimal function, then we must have equality for H\"older's inequality. This implies that the support of $f$ is $\overline{O}\cap\overline{\Om}$ (see the proof of \cite[Proposition $6$]{CNV}).

Then, the argument in \cite[Proof of Proposition $6$]{CNV} shows that $D^2_{\mathcal{D}'}\varphi$ has no singular part on $O$ (here, we replace $0$ in the argument of \cite{CNV} by $x_0\in O\cap \Om$).

Finally, we must have equality in the arithmetic-geometric inequality
$$\left(\frac{\partial_n\varphi}{x_n}\right)^{\frac{a}{n_a}} \left(\det D_{\mathcal{D}'}^2\varphi\right)^{\frac1{n_a}}\leq \frac1{n_a} L_{\mathcal{D}'}\varphi.$$

\item(i) If $a > 0$, then $D_{\mathcal{D}'}^2\varphi = \lambda\, \id_n$ and $\frac{\partial_n\varphi}{x_n} =\lambda$ for some $\lambda >0$, hence for every $x=(\overline x, x_n)\in \Om\cap O$, $\nabla\varphi(x) =\lambda (\x-\x_0,x_n)$ for some $\x_0\in \R^{n-1}$ and the (interior) of the support of $f^{p^\ast}$ is  $\Om$  (see \cite{CNV}). This  means that $f(x)=h_{p,a}(\lambda(\x-\x_0,x_n))$ as claimed.

\item(ii) If $a =0$, then $D_{\mathcal{D}'}^2\varphi = \lambda\, \id_n$ for some $\lambda > 0$, hence $\nabla\varphi(x) =\lambda (\x-\x_0,x_n -t_0)$ on the support of $f$ for some $(\x_0,t_0)\in \R^n$. Since the Brenier map sends the interior of the support of $f^{p^\ast}$ to the interior of support of $h_{p,a}^{p^\ast}$, we must have $\textrm{int(supp $f$)}=\Omega +  t_0 e_n$ with $t_0\ge 0$ and $f(x)=\mathbf1_{\Omega +  t_0 e_n}(x) \cdot h_{p,a}(\lambda(\x-\x_0,x_n-t_0))$. But since $f\in \dot{W}^{1,p}(\Om,\om)$ it forces $t_0=0$, as wanted.
\end{proof}

\emph{Proof of Proposition~\ref{mainL1}\text}: Proposition~\ref{mainL1} follows from the following proposition:
\begin{proposition}\label{p=1}
If $f\not=0$ is a $C^1$, compactly supported function on $\Om$, then
\begin{equation}\label{eq:L1case}
\frac{||\nabla f||_{L^1(\Om,\om)}}{||f||_{L^{\frac{n_a}{n_a-1}}(\Om,\om)}} \geq n_a\left(\int_B x_n^a \,dx\right)^{\frac1{n_a}}.
\end{equation}
This inequality extends to functions having $\om-$bounded variation, with equality if $f = h_{1,a}$. 
\end{proposition}
\begin{proof}
As in the proof of Proposition \ref{plarge1}, $f$ can be  assumed to be nonnegative and such that $||f||_{L^{\frac{n_a}{n_a-1}}(\Om,\om)}=1$. Denote $F = f^{\frac{n_a}{n_a-1}}$ and $G = h_{1,a}$ and let $\nabla\varphi$ be the Brenier map pushing $F\,\om \, dx$ forward to $G\,\om\, dx$. Then applying Lemma \ref{technical} with $\gamma=1-\frac1{n_a}$, we get
\begin{equation}\label{eq:L1sob}
n_a\left(\int_B x_n^a\, dx\right)^{\frac1{n_a}} \leq \int_\Om \nabla f\cdot (-\nabla\varphi)\, \om\, dx.
\end{equation}
Since $\nabla\varphi \in B$, then $||-\nabla\varphi|| \leq 1$, this implies
$$n_a \left(\int_B x_n^a\, dx\right)^{\frac1{n_a}} \leq \int_\Om ||\nabla f||_*\, \om\, dx = ||\nabla f||_{L^1(\Om,\om)}.$$
By an approximation argument, we can extend inequality~\eqref{eq:L1case} to functions having $\om-$bounded variation. And if $f =h_{1,a}$ then
\begin{align*}
\int_\Om \om \,||\nu_f||_*\, d(|Df|) & = \left(\int_{B\cap \Om}x_n^a \,dx\right)^{-\frac{n_a-1}{n_a}}\,\int_\Om \om\, ||\nu_B||_*\, d(|D\ind_B|)\\
&=\left(\int_{B\cap \Om}x_n^a \,dx\right)^{-\frac{n_a-1}{n_a}}\, P_{\om,||\cdot||}(B,\Om)\\
& = n_a\left(\int_{B\cap \Om} x_n^a\, dx\right)^{\frac1n_a}.
\end{align*}
This shows that equality holds in~\eqref{eq:L1case}.
\end{proof}


{\bf Proof of Theorem \ref{maintheorem2}}: As before, Theorem \ref{maintheorem2} will follow from the following duality principle. 

\begin{proposition}\label{GNineq}
Let $a\geq 0$, $p\in (1,n_a)$ and $\alpha \in (0,\frac{n_a}{n_a -p}]$, $\al\not=1$. Let $f\in \dot{W}^{1,p}(\Om,\om)$ and $g\in L^{\alpha p}(\Om,\om)$ be such that $||f||_{L^{\alpha p}(\Omega,\omega)}=||g||_{L^{\alpha p}(\Omega,\omega)} =1$. Then, for all $\mu >0$
\begin{align}\label{eq:nonhomogeneousGNineq}
\frac{\alpha p}{(\alpha -1)p_\alpha}&\int_\Omega |g|^{p_\alpha}\,\om\, dy -\frac{\mu^q}{q}\int_\Omega |g(y)|^{\alpha p}\, ||y||^q\,\omega(y)\,dy\notag\\
&\leq \frac{\alpha p -n_a(\alpha -1)}{(\alpha -1)p_\alpha}\int_\Omega |f(x)|^{p_\alpha}\, \omega(x)\, dx +  \frac{1}{p\mu^p}\int_\Omega||\nabla f(x)||_*^p\, \omega(x)\, dx,
\end{align}
where 
$$ p_\alpha =\alpha p - \alpha +1.$$
When $\mu =\mu_p :=q^{\frac{1}{q}}$, then equality in~\eqref{eq:nonhomogeneousGNineq} holds if $f = g = h_{\alpha,p,a}$.

As immediate consequences we have
\item(i) The dual principle
\begin{align}\label{eq:dualforGN}
&\sup_{||g||_{L^{\alpha p}(\Om,\om)=1}}\frac{\alpha p}{(\alpha -1)p_\alpha}\int_\Omega |g|^{p_\alpha}\,\om\, dy -\frac{\mu_p^q}{q}\int_\Omega |g(y)|^{\alpha p}\, ||y||^q\,\omega(y)\,dy\notag\\
&= \inf_{||f||_{L^{\alpha p}(\Om,\om)=1}}\frac{\alpha p -n_a(\alpha -1)}{(\alpha -1)p_\alpha}\int_\Omega |f|^{p_\alpha}\, \om\, dx +  \frac{1}{p\mu_p^p}\int_\Omega||\nabla f||_*^p\, \om\, dx,
\end{align}
and $h_{\alpha,p,a}$ is extremal in both variational problems.
\item(ii) The sharp weighted GN inequality: if $0\not= f\in \dot{W}^{1,p}(\Om,\om)$, then

for $\alpha > 1$,
\begin{equation}\label{eq:GNpositive}
\frac{||\nabla f||_{L^p(\Om,\om)}^\theta ||f||_{L^{p_\alpha}(\Om,\om)}^{1-\theta}}{||f||_{L^{\alpha p}(\Om,\om)}}\geq ||\nabla h_{\alpha,p,a}||_{L^p(\Om,\om)}^\theta ||h_{\alpha,p,a}||_{L^{p_\alpha}(\Om,\om)}^{1-\theta},
\end{equation}
where $\theta$ is given by~\eqref{eq:theta}.

for $\alpha < 1$,
\begin{equation}\label{eq:GNcompact}
\frac{||\nabla f||_{L^p(\Om,\om)}^\theta ||f||_{L^{\alpha p}(\Om,\om)}^{1-\theta}}{||f||_{L^{p_\alpha}(\Om,\om)}}\geq \frac{||\nabla h_{\alpha,p,a}||_{L^p(\Om,\om)}^\theta}{ ||h_{\alpha,p,a}||_{L^{p_\alpha}(\Om,\om)}},
\end{equation}
where $\theta$ is given by~\eqref{eq:theta1}.
\end{proposition}
\begin{proof}
As in the proof of Proposition \ref{plarge1}, we can assume that $f,g$ are nonnegative, compactly supported functions on $\overline{\Om}$, and that $f$ is $C^1(\overline{\Om})$. Let $\nabla\varphi$ be the Brenier map pushing $f^{\alpha p}\, \om\, dx$ forward to $g^{\alpha p}\, \om\, dx$. Applying Lemma \ref{technical} to 
$$F = f^{\alpha p},\quad\, G =g^{\alpha p},\quad\,\text{and}\,\quad \gamma =\frac{p_\alpha}{\alpha p},$$
we obtain
\begin{align}\label{eq:GNdual}
\frac{\alpha p}{\alpha-1}\,\int_\Om g^{p_\alpha}\, \om\, dx& \leq \frac{\alpha p -n_a(\alpha-1)}{\alpha -1}\, \int_\Om f^{\alpha p}\, \om\, dx\notag\\
&\,\,\,\,\,\, - p_\alpha \int_\Om \nabla f\cdot f^{\alpha(p-1)}\nabla\varphi\, \om\, dx.
\end{align} 
For all $\mu > 0$, applying Young's inequality~\eqref{eq:Youngineq}, we get
\begin{equation}\label{eq:GNdual1}
-\int_\Om \nabla f\cdot f^{\alpha(p-1)}\nabla\varphi\, \om\, dx \leq \frac{1}{p\mu^p} \int_\Omega||\nabla f||_*^p\, \om\, dx + \frac{\mu^q}{q}\int_\Omega |f|^{\alpha p}\, ||\nabla\varphi||^q\,\om\,dx.
\end{equation}
Plugging~\eqref{eq:GNdual1} into~\eqref{eq:GNdual}, and using the definition of mass transportation, we get~\eqref{eq:nonhomogeneousGNineq}.

When $f = g =h_{\alpha,p,a}$, we must have $\nabla\varphi(x) = x$. As in the proof of Proposition \ref{plarge1}, we have an equality in~\eqref{eq:GNdual}. Moreover, if $f = h_{\alpha,p,a}$,
$$\nabla f(x) = -q\left(\sigma_{\alpha,p,a} + (\alpha-1)||x||^q\right)_+^{\frac\alpha{\alpha-1}} ||x||^{q-1}x^*,$$
then
$$-\nabla f(x)\, \cdot \, f(x)^{\alpha p} \,\nabla\varphi(x) =q\, f(x)^{\alpha p}\, ||x||^q,$$
since $\nabla\varphi(x) =x$. We also have
$$||\nabla f(x)||_*^p = q^p\, f(x)^{\alpha p}\, ||x||^q,$$
and, for $\mu =\mu_p:= q^{\frac1q}$, we have
\begin{align*}
-\nabla f(x)\, \cdot \, f(x)^{\alpha p} \,\nabla\varphi(x)&=\frac{q}{p}\, f(x)^{\alpha p}\, ||x||^q + f(x)^{\alpha p}\, ||x||^q\\
&=\frac1{p\mu^p}\,||\nabla f(x)||_*^p + \frac{\mu^q}{q}\, f(x)^{\alpha p}\, ||\nabla\varphi(x)||^q.
\end{align*}
This shows that~\eqref{eq:GNdual1} becomes an equality, and then we have equality in~\eqref{eq:nonhomogeneousGNineq} when $f =g= h_{\alpha,p,a}$ and $\mu = \mu_p$.

The part $(i)$ is implied by~\eqref{eq:nonhomogeneousGNineq}. The proof of part $(ii)$ is similar the one of Theorem 4 in \cite{CNV}. This finishes the proof of Proposition \ref{GNineq}. 
\end{proof}


\section{Deriving sharp GN inequalities on Euclidean space}

In this section, we use Proposition~\ref{maintheorem1} to derive part of, and to extend, a family of GN inequalities on $\R^n$ obtained in \cite{CNV, DD1, DD2}. The idea is to use the weighted Sobolev inequalities on $\R^n\times \R_+$ for the weight $\om(x) = x_{n+1}^a$ to derive a GN inequalities on $\R^n$ (for the usual Lebesgue measure). For the case $p=2$ and the Euclidean norm, this argument is well known (we learned it from D.~Bakry) and is recalled in~\cite{BGL}. 
With the $L^p$-weighted Sobolev in hand, it remains to adapt the argument to $p\neq 2$. 

We will also get  GN inequalities for arbitrary norms on $\R^n$. But note that even if one is interested only in the Euclidean norm $|\cdot|$ on $\R^n$, we need to use the weighted Sobolev inequality with a non-Euclidean norm $||\cdot||$ on $\R^{n+1}$ in the case $p\not=2$, namely 
$$||(x,t)|| :=\left(|x|^q + |t|^q\right)^{\frac1q},$$
with $q = \frac{p}{p-1}$. 

So let $||\cdot||$ be a norm on $\R^n$; its dual norm is denoted by $||\cdot||_*$. If $f\in \dot{W}^{1,p}(\R^n)$, we denote
$$||\nabla f||_{L^p(\R^n)} = \left(\int_{\R^n} ||\nabla f(x)||_*^p\, dx\right)^{\frac1p}.$$
We also denote here 
$$||f||_r = \left(\int_{\R^n} |f(x)|^r dx\right)^{\frac1r}$$
for $r\in \R\setminus\{0\}$.  We have:

\begin{theorem}\label{FromSobtoGNS}
Let $n\geq 1$,  $a\geq 0$, and  set $(n+1)_a = n+1+a$ as above and $ \alpha = \frac{np+a+1}{pn+a+1 -p^2}$. For any function $f\in \dot{W}^{1,p}(\R^n)$, we have

\item(i) If $1 <p <\frac{n+\sqrt{n^2 +4(1+a)}}{2}$, then
\begin{equation}\label{eq:CNV}
||f||_{\alpha p} \leq GN(n,p,a)||\nabla f||^\theta_{L^p(\R^n)}||f||_{\alpha p -\alpha +1}^{(1-\theta)},
\end{equation}
where $\theta\in (0,1)$ is equal to $\frac{n(\alpha -1)}{\alpha(np -(\alpha p+1-\alpha)(n-p))}$. The constant $GN(n,a,p)$ is optimal, and is given by
\begin{equation}\label{eq:bestconstant11}
 GN(n,p,a) =
\left[\frac{y(\alpha-1)^p}{q^{p-1}n}\right]^{\frac{\theta}p}\left[\frac{qy - n}{qy}\right]^{\frac{1}{\alpha p}}\left[\frac{\Gamma(y)}{\vol(K)\Gamma(y-\frac{n}{q})\Gamma(\frac{n}q +1)}\right]^{\frac{\theta}{n}},
\end{equation}
where $y = n +\frac{a+1-p}{p}$, and $K = \{x\in \R^n\, : \, ||x||\leq 1\}$.

\item(ii) If $\frac{n+\sqrt{n^2 +4(1+a)}}{2} < p <(n+1)_a$, then
\begin{equation}\label{eq:inverseCNV}
||f||_{\alpha p -\alpha +1}\leq GN(n,p,a)||\nabla f||_{L^p(\R^n)}^\theta ||f||_{\alpha p}^{1-\theta},
\end{equation}
where $\theta \in (0,1)$ is equal to $\frac{n(1-\alpha)}{(\alpha p -\alpha +1)(n-\alpha(n-p))}$. The constant $GN(n,a,p)$ is optimal and is given by
\begin{equation}\label{eq:bestconstant12}
GN(n,p,a)= \left[\frac{z(1-\alpha)^p}{q^{p-1}n}\right]^{\frac\theta{p}}\, \left[\frac{qz}{qz -n}\right]^{\frac{1-\theta}{\alpha p}}\, \left[\frac{\Gamma(z)}{\vol(K)\,\Gamma(\frac{n}q+1)\,\Gamma(z -\frac{n}q)}\right]^{\frac{\theta}{n}} ,
\end{equation}
where $z = -\frac{\alpha p}{1-\alpha} -1 = n+\frac{a+1-p}p$, and $K =\{x\in \R^n\, : \, ||x||\leq 1\}$.

Moreover, equality in~\eqref{eq:CNV} and~\eqref{eq:inverseCNV} holds if and only if
$$ f(x) = c\, (1+||\lambda(x-x_0)||^q)^{-\frac{1}{\alpha -1}},$$
for some $c\in \R$, $0\not=\lambda \in \R$ and $x_0\in \R^n$.
\end{theorem}
Let us comment the results of Theorem \ref{FromSobtoGNS}:  
\begin{enumerate}
\item  In the case $(i)$ we have $\alpha \geq 0$, and moreover, if $p < n$ then
$$ 1< \alpha =\frac{np +a +1}{np+a+1 -p^2}\leq \frac{np +1}{np+ 1 -p^2}< \frac{n}{n-p}, \quad \forall\, a\geq 0. $$
More precisely, when $p<n$, $\alpha$ is a continuous function of $a$ that takes all the values between $1$ and $\frac{np +1}{np+ 1 -p^2}$, so we get (in the case of an arbitrary norm due to~\cite{CNV}) the family~\eqref{eq:DelpinoDolbeault} except for the values of $\alpha$ between $\frac{np +1}{np+ 1 -p^2}$ and $\frac{n}{n-p}$. 
Let us also mention that there is an inequality for the family of  $0< \alpha < 1$:
$$\left(\int_{\R^n}|f|^{\alpha(p-1)+1}dx\right)^{\frac1{\alpha (p-1)+1}}\leq G(n,p,\alpha)\left(\int_{\R^n}|\nabla f|^p dx\right)^{\frac\theta p} \left(\int_{\R^n}|f|^{\alpha p} dx\right)^{\frac{1-\theta}{\alpha p}},$$
where $\theta$ is determined by scaling invariance. The extremal functions in this case are compactly supported. We do not know how to derive these remaining cases with our approach.

\item The second remark is that our inequality~\eqref{eq:CNV} holds even when $p >n$. Indeed, since $\frac{n+\sqrt{n^2 +4(1+a)}}{2} > n$, the range in the statement $(i)$ of the Theorem contains a range where $p> n$.  This is different than the usual condition $p < n$ required in the classical GN inequalities, and seems to be new.

\item The case $(ii)$ of the Theorem concerns also  the range where $p>n$. However, note that inequality~\eqref{eq:inverseCNV} is a sharp GN type inequalities involving to $L^r-$ norm of functions with $r < 0$, since $\alpha p$ and $\alpha p -\alpha +1$ are both negative in this case.

\item Theorem \ref{FromSobtoGNS} is true in the dimension $1$. In this case, inequality~\eqref{eq:CNV} gives us a subfamily of sharp GN inequalities on the real line due to Agueh \cite{A1, A2}. In these two papers, Agueh investigated the sharp constants and optimal functions of GN inequalities involving the $L^p$ norm of the gradient by studying a $p-$Laplacian type equation. Indeed, the link between the sharp constants and mass transportation theory suggests a special change of functions that brings us to the solution of $p-$Laplacian type equations, in all generality when the dimension is $1$ and in some particular cases when $n >1$.

\item Finally, note that our method will allow us to characterize all cases of equality in the stated GN inequalities.

\end{enumerate}
\begin{proof}
As before, we denote $\Om = \R^n\times \R_+$ and $q = p/(p-1)$. Suppose $h$ is a nonnegative, smooth function on $\R^n$ such that $h(x)\to \infty$ when $|x|\to \infty$. We define a new function on $\Om$ by
$$ g(x,t) = \bigl(h(x)+t^q\bigl)^{-\frac{(n+1)_a-p}{p}},\quad (x,t)\in \Om.$$
We also define a new norm on $\R^{n+1}$ by
$$|||(x,t)||| = \left(||x||^q +|t|^q\right)^{\frac1q}.$$
Its dual norm is given by
$$|||(y,s)|||_* = \left(||y||_*^p +|s|^p\right)^{\frac1p}. $$
Applying Proposition \ref{maintheorem1} to $\Om = \R^n\times \R_+$ with norm $|||\cdot|||$ and weight $\om(x,t) =t^a$, we get
\begin{equation}\label{eq:GN1}
\bigg(\int_{\Om} g(x,t)^{\frac{(n+1)_ap}{(n+1)_a-p}}\,t^a\,dx\,dt\bigg)^{\frac{(n+1)_a-p}{(n+1)_a }}\leq S(n+1,a,p)\int_{\Om}|||\nabla g(x,t)|||_*^p\, t^a \,dx\,dt,
\end{equation}
and the equality holds true if and only if
$$ h(x) = a+ ||x -x_0||^q $$
for some $a >0$ and $x_0\in \R^n$. We have
\begin{align}\label{eq:GN2}
\int_{\Om} g(x,t)^{\frac{(n+1)_ap}{(n+1)_a-p}}\,t^a\, dx\,dt&=\int_{\Om} (h(x)+t^q)^{-(n+1)_a}\,t^a\,dx\,dt\notag\\ 
&=\int_0^\infty (1+t^q)^{-n-1-a}\,t^a\,dt\, \int_{\R^n}h(x)^{-n-\frac{a+1}{p}}\,dx \notag\\
&=S_1(n,a,p)\, \int_{\R^n}h(x)^{-n-\frac{a+1}{p}}\,dx .
\end{align}
Since
$$ \nabla g(x,t) = -\frac{(n+1)_a-p}{p}(h(x)+t^q)^{-\frac{(n+1)_a}{p}}(\nabla h(x), qt^{q-1}), $$
so
$$ \int_{\Om}|||\nabla g(x,t)|||_*^p\, t^a\, dx\,dt =\frac{((n+1)_a -p)^p}{p^p}\int_{\Om}(h(x)+t^q)^{-(n+1)_a}\left(||\nabla h(x)||_*^p+q^pt^q\right)\,t^a\,dx\,dt.$$
We now have
\begin{align}\label{eq:GN3}
\int_{\Om}(h(x)+t^q)^{-(n+1)_a}||\nabla h(x)||_*^p\,t^a\,dx\,dt& =\int_0^\infty\frac{t^a}{(1+t^q)^{n_a}} dt \, \int_{\R^n}||\nabla h(x)||_*^ph(x)^{-n-\frac{1+a}{p}}\,dx\notag\\ 
& = S_2(n,a,p)\, \int_{\R^n}||\nabla (h(x)^{-\frac{n}{p}-\frac{a+1}{p^2} +1})||_*^p\,dx,
\end{align}
and
\begin{align}\label{eq:GN4}
\int_{\Om} (h(x)+t^q)^{-n_a}\,t^{q+a}\,dx\,dt&=\int_0^\infty\frac{t^{q+a}}{(1+t^q)^{-n_a}}\,dt\, \int_{\R^n}h(x)^{-n-\frac{a+1 -p}{p}}\,dx\notag\\ 
&= S_3(n,a,p)\, \int_{\R^n}h(x)^{-n-\frac{a+1-p}{p}}\,dx.
\end{align}
Combining~\eqref{eq:GN1},~\eqref{eq:GN2},~\eqref{eq:GN3}, and~\eqref{eq:GN4}, we get
\begin{align}\label{eq:GN5}
\bigg(\int_{\R^n} h(x)^{-n-\frac{1+a}{p}}\,dx\bigg)^{\frac{n+1+a-p}{n+a+1}}&\leq A(n,a,p)\int_{\R^n} ||\nabla(h(x)^{-\frac{np+a+1-p^2}{p^2}})||_*^p\,dx\notag\\
& \quad\quad + B(n,a,p)\int_{\R^n} h(x)^{-n-\frac{a+1-p}{p}}\,dx,
\end{align}
where $A(n,a,p)$ and $B(n,a,p)$ are constants depending only on $n$, $a$ and $p$. Moreover, there is equality in~\eqref{eq:GN5} if and only if
$$ h(x) = a + ||x-x_0||^q,$$
for some $a >0$ and $x_0\in \R^n$. Changing $h$ to $f^{-\frac{p^2}{np+a+1 -p^2}}$ yields the following inequality:
\begin{equation}\label{eq:GN6}
\left(\int_{\R^n} f^{\alpha p}\,dx\right)^{\frac{n+a+1-p}{n+a+1}}\leq A(n,a,p)\int_{\R^n} ||\nabla f||_*^p\,dx + B(n,a,p)\int_{\R^n} f^{\alpha(p-1)+1}\,dx.
\end{equation}
The above changement of functions implies that equality in~\eqref{eq:GN6} holds if and only if
$$f(x) = (a+||x-x_0||^q)^{-\frac{1}{\alpha-1}},$$
for some $a>0$ and $x_0\in\R^n$. Inequality~\eqref{eq:GN6} is a nonhomogeneous form of GN inequalities. 
\begin{enumerate}
\item If $1 <p < \frac{n+\sqrt{n^2+4(1+a)}}2$ then $\alpha > 1$. Applying~\eqref{eq:GN6} to functions $\lambda f$, $\lambda > 0$, and optimizing over $\lambda > 0$ yields the following inequality
\begin{equation}\label{eq:DD}
||f||_{\alpha p} \leq C(n,a,p)\, ||\nabla f||_{L^p(\R^n)}^{\theta}\, ||f||_{\alpha (p-1)+1}^{1-\theta}.
\end{equation}
Changing $f$ by $f_\lambda(x) = f(\frac{x}\lambda)$ with $\lambda > 0$, we must have 
$$\theta = \frac{n(\alpha -1)}{\alpha(np -(\alpha p+1-\alpha)(n-p))}.$$
From the above proof, we see that $f(x) = (1+||x||^q)^{-\frac{1}{\alpha-1}}$ is an extremal function for~\eqref{eq:DD}. Then this inequality is optimal, and hence $C(n,a,p)$ is the best constant. 

\item If $\frac{n+\sqrt{n^2+4(1+a)}}2 < p < (n+1)_a$ then $\alpha < 0$. It is easy to check that $\alpha p- \alpha + 1 <0$ in this case. Applying~\eqref{eq:GN6} to functions $\lambda f$, $\lambda > 0$, and optimizing over $\lambda >0$ yields the following inequality
\begin{equation}\label{eq:DDDD}
||f||_{\alpha p -\alpha +1}\leq D(n,a,p)||\nabla f||_{L^p(\R^n)}^\theta ||f||_{\alpha p}^{1-\theta}.
\end{equation}
Changing $f$ by $f_\lambda(x) =f(\frac{x}\lambda)$ with $\lambda >0$, we must have
$$\theta = \frac{n(1-\alpha)}{(\alpha p -\alpha +1)(n-\alpha(n-p))}.$$
As above, we see that $f(x) = (1+||x||^q)^{-\frac{1}{\alpha-1}}$ is an extremal function for~\eqref{eq:DDDD}. Then this inequality is optimal, and hence $D(n,a,p)$ is the best constant.
\end{enumerate}

From the proof above, we see that~\eqref{eq:GN6} is equivalent to~\eqref{eq:DD} in the case $(i)$ (and~\eqref{eq:DDDD} in the case $(ii)$). Moreover, if $f$ is an optimal function to~\eqref{eq:DD} (also to~\eqref{eq:DDDD}) then $\lambda f$ is an optimal function to~\eqref{eq:GN6} for some $\lambda >0$. This shows that $f$ is of the form announced in the theorem.

A direct computation using function $f(x) = \left(1 +||x||^q\right)^{-\frac{1}{\alpha-1}}$ shows that the best constants $C(n,a,p)$ and $D(n,a,p)$ are given by~\eqref{eq:bestconstant11} and~\eqref{eq:bestconstant12} respectively.
\end{proof}

\section{A Generalization to $\R^{n-m}\times\R^m_+$ and application}
In this section, we denote $\Sigma =\R^{n-m}\times \R^m_+$ with $ n\geq m$ and $m\geq 1$. An element of $\Sigma$ is written as
$$(x,t) =(x_1,\cdots, x_{n-m}, t_1,\dots,t_m),\quad\, x_1\cdots, x_{n-m}\in \R^{n-m},\quad t_1,\cdots,t_m >0.$$
We consider a monomial weight $\sigma$ on $\Sigma$ of the form
$$\sigma(x,t) = t_1^{a_1}\cdots t_m^{a_m},$$
where $a_1,\cdots, a_m\geq 0$. For such $a_1,\cdots,a_m$, we denote $n_a =n + a_1+\cdots +a_m$ the corresponding fractional dimension of $(\Sigma,\sigma)$. For $1\leq p < n_a$, we denote 
$$p^* =\frac{n_ap}{n_a -p}.$$
As in the introduction, we denote, for $p \geq 1$, $\dot{W}^{1,p}(\Sigma,\sigma)$ the space of all measurable functions $f$ on $\Sigma$ such that its level sets $\{(x,t)\in \Sigma \, :\, |f(x,t)| > a\}$, $a >0$, have finite measure with respect to measure of density $\sigma$ on $\Sigma$, and its distributional gradient $\nabla f$ belongs to $L^p(\Sigma,\sigma)$. 

Let $||\cdot||$ be a norm on $\R^n$, and let $||\cdot||_*$ be its dual norm. For $f\in \dot{W}^{1,p}(\Sigma,\sigma)$, we define
$$||\nabla f||_{L^p(\Sigma,\sigma)} =\left(\int_\Sigma ||\nabla f(x,t)||_*^p\, \sigma(x,t)\, dxdt\right)^{\frac1p}.$$
Let us denote $B = \{(x,t)\in \R^n\, :\, ||(x,t)||\leq 1\}$, and introduce the functions
\begin{equation}\label{eq:exfun}
h_p (x,t) =
\begin{cases}
\left(\sigma_p + ||(x,t)||^q\right)^{\frac{n_a-p}{p}}&\mbox{ if $1 < p < n_a$}\\
\frac{\ind_{B\cap\Sigma}(x,t)}{\left(\int_{B\cap \Sigma} \sigma\, dxdt\right)^{\frac{n_a-1}{n_a}}}&\mbox{ if $p=1$},
\end{cases}
\end{equation}
where $\sigma_p$ is chosen such that
$$\int_\Sigma h_p(x,t)^{p^*}\, \sigma(x,t) \, dx\,dt = 1.$$
Unlike before,  we drop the indices relating to $m$ and $a_1,\cdots, a_m$, and norm $||\cdot||$ in the notation of these extremal functions.

As in the introduction, we can define the notion of function having $\sigma-$bounded variation on $\Sigma$, and define the weighted perimeter of a subset of $\R^n$ with respect to $\sigma$. 

We can now state a generalization of Propositions~\ref{maintheorem1} and~\ref{mainL1} which is already presented in~\cite{CR}, as discussed in the introduction. Actually, our proof gives a bit more, namely a duality principle analogue to the one in Proposition~\ref{plarge1}.

\begin{proposition}\label{maintheorem1'}
\item (i) Let $a_1,\cdots,a_m \geq 0$. If $1< p < n_a$, there exists a constant $S(n,m,a,p) $ such that for any  $f\in \dot{W}^{1,p}(\Sigma,\sigma)$, 
\begin{equation}\label{eq:swSobpm}
||f||_{L^{p^*}(\Sigma,\sigma)}\leq S(n,m,a,p)||\nabla f||_{L^p(\Sigma,\sigma)},
\end{equation}
where $p^* = \frac{n_a p}{n_a -p}$. The best constant $S(n,m,a,p)$ is given by
\begin{equation*}\label{eq:bestconstantm}
S(n,m,a,p) = \left(\frac{(p-1)^{p-1}}{n_a(n_a-p)^{p-1}}\right)^{\frac1{p}}\left[\frac{\Gamma(\frac{n_a}p)\Gamma(\frac{n_a(p-1)}{p}+1)}{\Gamma(n_a)} \int_{B\cap \Sigma} \sigma(x,t) \,dx\, dt\right]^{-\frac1{n_a}}
\end{equation*}
Equality in~\eqref{eq:swSobpm} holds if and only if
$$ f(x,t) = c\,h_{p}(\lambda(x-x_0,t)),$$
for some $c\in \R$, $\lambda >0$ and $x_0\in \R^{n-m}$.

\item (ii) When $p = 1$, there is a constant $S(n,m,a,1)$ such that for any smooth compactly supported function $f$, then 
\begin{equation}\label{eq:swSob1m}
||f||_{L^{\frac{n_a}{n_a -1}}(\Sigma,\sigma)}\leq S(n,m,a,1) ||\nabla f||_{L^1(\Sigma,\sigma)}.
\end{equation}
The best constant $S(n,m,a,1)$ is given by
\begin{equation*}\label{eq:bestconstant1m}
S(n,m,a,1) = n_a^{-1}\left(\int_{B\cap \Sigma} \sigma(x,t)\, dx\, dt\right)^{-\frac1{n_a}}.
\end{equation*}
The inequality~\eqref{eq:swSob1m} extends to all functions with $\sigma-$bounded variation, equality in~\eqref{eq:swSob1m} holds if  for some $c\in \R$, $\lambda > 0$ and $x_0\in \R^{n-m}$,
$$ f(x,t) = c\, h_{1}(\lambda(x-x_0,t)).$$
\end{proposition}
 
As before, the case $p=1$ in Proposition \ref{maintheorem1'} is equivalent to a weighted isoperimetric inequality on $\Sigma$, that is, among all subsets $E$ of $\R^n$ such that $\int_{E\cap \Sigma} \sigma(x,t)\, dxdt $ is equal to $\int_{B\cap \Sigma}\sigma(x,t)\, dxdt$ then $B$ has smallest weighted perimeter.

To generalize Theorem \ref{maintheorem2}, we introduce the following family of functions, for $0< \alpha \not=1$,
\begin{equation*}
h_{p,\alpha}(x,t) = \left(\sigma_p + (\alpha-1)||(x,t)||^q\right)_+^{\frac1{1-\alpha}},
\end{equation*}
where $\sigma_p$ is chosen such that
$$\int_\Sigma h_{p,\alpha}(x,t)^{\alpha p}\, \sigma(x,t)\, dx\, dt =1.$$
We now can state a generalization of Theorem \ref{maintheorem2} as follows:
\begin{theorem}\label{maintheorem2'}
Let $a_1,\cdots, a_m\geq 0$, $p\in (1,n_a)$ and $\alpha\in (0,\frac{n_a}{n_a -p}]$, $\alpha \not= 1$.
\item (i) If $\alpha >1$, there exists a constant $G_{n,m,a}(\alpha,p)$ such that for any $f\in \dot{W}^{1,p}(\Sigma,\sigma)$,
\begin{equation}\label{eq:swGNineqm}
||f||_{L^{\alpha p}(\Sigma,\sigma)} \leq G_{n,m,a}(\alpha,p)||\nabla f||_{L^p(\Sigma,\sigma)}^\theta ||f||_{L^{\alpha(p-1)+1}(\Sigma,\sigma)}^{1-\theta},
\end{equation} 
where
\begin{equation*}\label{eq:thetam}
 \theta =\frac{n_a(\alpha-1)}{\alpha(n_ap -(\alpha p+1 -\alpha)(n_a -p))} =\frac{p^*(\alpha -1)}{\alpha p(p^* -\alpha p+\alpha -1)}, 
\end{equation*}
the best constant $G_{n,m,a}(\alpha,p)$ takes the explicit form, denoting $y =\frac{\alpha(p-1) +1}{\alpha -1}$
\begin{equation*}\label{eq:explicitconstantm}
G_{n,m,a}(\alpha,p) = \left[\frac{y(\alpha-1)^p}{q^{p-1}n_a}\right]^{\frac{\theta}p}\left[\frac{qy - n_a}{qy}\right]^{\frac{1}{\alpha p}}\left[\frac{\Gamma(y)}{\Gamma(y-\frac{n_a}{q})\Gamma(\frac{n_a}{q}+1)\int_{B\cap \Sigma} \sigma\, dx dt}\right]^{\frac{\theta}{n_a}}.
\end{equation*}
Equality in~\eqref{eq:swGNineqm} holds if 
$$ f(x) = c\, h_{p,\alpha}(\lambda(x-x_0,t)),$$
for some $c\in \R$, $\lambda >0$ and $x_0\in \R^{n-m}$.
\item (ii) If $\alpha < 1$, there exists a constant $N_{n,m,a}(\alpha,p)$ such that for any $f\in \dot{W}^{1,p}(\Omega,\omega)$
\begin{equation}\label{eq:swGNineq1m}
||f||_{L^{\alpha (p -1) +1}(\Sigma,\sigma)}\leq N_{n,m,a}(\alpha,p)||\nabla f||_{L^p(\Sigma,\sigma)}^\theta ||f||_{L^{\alpha p}(\Sigma,\sigma)}^{1-\theta},
\end{equation}
where
\begin{equation*}\label{eq:theta1m}
\theta =\frac{n_a(1-\alpha)}{(\alpha p+1-\alpha)(n-\alpha(n-p))} = \frac{p^*(1-\alpha)}{(p^*-\alpha p)(\alpha p+ 1-\alpha)},
\end{equation*}
the best constant $N_{n,m,a}(\alpha,p)$ takes the explicit form, denoting $z = \frac{\alpha p -\alpha +1}{1 -\alpha }$,
\begin{equation*}\label{eq:explicitconstant1m}
N_{n,m,a}(\alpha,p) = \left[\frac{z(1-\alpha)^p}{q^{p-1}n_a}\right]^{\frac{\theta}p}\left[\frac{qz}{qz+n_a}\right]^{\frac{1-\theta}{\alpha p}}\left[\frac{\Gamma(z+1+\frac{n_a}{q})}{\Gamma(z+1)\Gamma(\frac{n_a}{q}+1)\int_{B\cap\Sigma}\sigma \, dx dt}\right]^{\frac{\theta}{n_a}}.
\end{equation*}
Equality in~\eqref{eq:swGNineq1m} holds if 
$$ f(x) = c\, h_{p,\alpha}(\lambda(x-x_0,t)),$$
for some $c\in \R$, $\lambda > 0$ and $x_0\in \R^{n-m}$.
\end{theorem}

The proofs of Proposition~\ref{maintheorem1'} and~Theorem~\ref{maintheorem2'} are similar to their companion stated on $\Omega$. The proof relies on the following lemma which is a generalization of Lemma \ref{technical}.

\begin{lemma}\label{technicalm}
Let $a_1,\cdots, a_m\geq 0$, and $1\not=\gamma \geq 1 -\frac1{n_a}$. Let $F$ and $G$ be two nonegative functions on $\Sigma$ with $\int_\Sigma F\,\sigma\, dx dt =\int_\Sigma G\, \sigma\, dx dt =1$, $F^\gamma$ is $C^1$ on $\Sigma$ and $F,G$ are compactly supported on $\overline{\Sigma}$. Then if  $\nabla\varphi$ is the Brenier map pushing $F\,\sigma\, dx dt$ forward to $G\,\sigma\, dx dt$, we have 
\begin{equation*}\label{eq:integraionbypartm}
\frac{1}{1-\gamma}\int_\Sigma G^\gamma \, \sigma \, dx dt\leq \frac{1-n_a(1-\gamma)}{1-\gamma}\int_\Sigma F^\gamma\, \sigma\, dx dt -\int_\Sigma \nabla F^\gamma\cdot \nabla\varphi\, \sigma\, dx dt.
\end{equation*}
\end{lemma}
This lemma is proved  in the same way as Lemma~\ref{technical}. When we arrive to  the step of justifying the integration by parts, we define the function $F_k$ by 
$$F_k(x,t) =F^\gamma(x,t)\, \theta_k(t_1)\cdots\theta_k(t_m),$$
where $\theta_k$ is defined in the proof of Lemma \ref{technical}.

We conclude this section by studying $L^p$-logarithmic Sobolev inequality announced in the introduction. More precisely,
we will explain how to use Proposotion~\ref{maintheorem1'} to prove Proposition~\ref{LplogSobineq}. 

Let $\Omega$, $a \geq 0$, $\om$ and $p\geq 1$ be as in Proposition \ref{LplogSobineq}. For $k\geq 1$, define
$$\Om_k = \underbrace{\Om\times\cdots \times \Om}_{k \text{ times}}.$$
An element of $\Om_k$ is written by $(\x^1,t_1,\cdots, \x^k, t_k)$ with $\x^1,\cdots, \x^k \in \R^{n-1}$ and $t_1, \cdots, t_k > 0$. Let $\om_k$ be the weight function on $\Om_k$ given by
$$\om_k(\x^1,t_1,\cdots, \x^k, t_k) = t_1^a \cdots t_k^a.$$
If $||\cdot||$ is a norm on $\R^n$, we define a new norm on $\R^{nk}$ by 
$$|||(x^1,\cdots,x^k)||| = \left(\sum_{i=1}^k ||x^i||^q\right)^{\frac1q}.$$
Let $B_k$ denote its unit ball. An easy computation shows that the dual norm of $|||\cdot|||$ is given by
$$|||(y^1,\cdots, y^k)|||_* = \left( \sum_{i=1}^k ||y^i||_*^p\right)^{\frac1p}.$$
Choosing $k$ such that $k(n+a) > p$, applying Proposition \ref{maintheorem1'} to the function
$$f_k(\x^1,t_1,\cdots,\x^k,t_k) =  f(\x^1,t_1)\times\cdots\times f(\x^k,t_k),$$
we obtain
\begin{equation}\label{eq:ApplSob}
\left(\int_\Om f^{\frac{kn_ap}{kn_a -p}}\, \om\, dx\right)^{\frac{kn_a-p}{n_a p}} \leq k^{\frac1p}S(kn,k,a,p)\,||\nabla f||_{L^p(\Om,\om)}.
\end{equation}
It is easy to prove that
$$\int_{B_k\cap \Om_k} \om_k =\frac{q}{k n_a}\left(\frac{n_a}q \int_{B\cap\Om} \om\, dx\right)^k  \frac{\Gamma(\frac{n_a}q)^k}{\Gamma(\frac{kn_a}q)}.$$
Using the Stirling's formula, we get
$\Gamma(x)^{\frac1x}\sim \frac{x}e$, and so
$$\lim_{k\to\infty} k^{\frac1p}\, S(nk,k,a,p) = \left[\frac{p}{n_a}\left(\frac{p-1}e\right)^{p-1}\right]^{\frac1p} \left(\Gamma(\frac{n_a}q +1)\int_{B\cap \Om} \om\, dx\right)^{-\frac1{n_a}}.$$
Taking logarithmic both sides of~\eqref{eq:ApplSob} and let $k\to\infty$, we get~\eqref{eq:LplogSob}.

Note that this proof can be used to generalize Proposition \ref{LplogSobineq} to the domain $\Sigma$ with the weight $\sigma$ given by~\eqref{eq:generalweight}.

\subsection*{Acknowledgment}
I am very grateful to my advisor Dario Cordero-Erausquin for his encouragements, his careful review of this manuscript, and for useful discussions on this problem.



\end{document}